\documentclass{amsart}
\pdfoutput=1
\usepackage[T1]{fontenc}
\usepackage{amsmath}
\usepackage{amsthm}
\usepackage{amsfonts}
\usepackage{amssymb}
\usepackage{enumerate}
\usepackage{mathrsfs}
\usepackage{graphicx}
\usepackage[active]{srcltx}
\usepackage{hyperref}
\usepackage{tikz}
\usetikzlibrary{matrix,arrows}

\newcommand\bp{{\mathbb P}}

\newcommand\bc{{\mathbb C}}

\newcommand\bn{{\mathbb N}}
\newcommand\bz{{\mathbb Z}}

\newcommand\bb{{\mathbb B}}

\newcommand\cC{{\mathcal C}}

\newcommand{\longtwoheadrightarrow}{\longrightarrow\hspace{-1.2em}\rightarrow\hspace{.2em}}

\newtheorem{thm}{Theorem}[section]

\newtheorem{prop}[thm]{Proposition}
\newtheorem{cor}[thm]{Corollary}
\newtheorem{lema}[thm]{Lemma}

\theoremstyle{remark}
\newtheorem{obs}[thm]{Remark}
\theoremstyle{definition}
\newtheorem{dfn}[thm]{Definition}

\newtheorem{cnv}[thm]{Convention}
\numberwithin{equation}{section}
\numberwithin{figure}{section}

\newcommand\enet[1]{\renewcommand\theenumi{#1}
\renewcommand\labelenumi{\theenumi}}

\title[Computation-free presentation of the fundamental
group]{Computation-free presentation of the fundamental
group of generic $(p,q)$-torus curves}
\author[E. Artal]{Enrique Artal Bartolo}
\address{Departamento de Matem\'aticas, IUMA\\ 
Universidad de Zaragoza\\ 
C.~Pedro Cerbuna 12\\ 
50009 Zaragoza, Spain} 
\email{artal@unizar.es, jicogo@unizar.es} 

\author[J.I. Cogolludo]{Jos{\'e} Ignacio Cogolludo-Agust{\'i}n}
\address{} 
\email{} 

\author[J. Ortigas]{Jorge Ortigas-Galindo}
\address{Centro Universitario de la Defensa-IUMA\\ 
Academia General
Militar\\ 
Ctra. de Huesca s/n\\ 
50090, Zaragoza, Spain} 
\email{jortigas@unizar.es} 

\begin{document}

\thanks{All authors are partially supported by
the Spanish Ministry of Education MTM2010-21740-C02-02.}  


\begin{abstract}
In this note, we present a new method for computing fundamental groups of curve complements 
using a variation of the Zariski-Van Kampen method on general ruled surfaces.
As an application we give an alternative (computation-free) proof for the fundamental group of 
generic $(p,q)$-torus curves.
\end{abstract}

\maketitle

\section*{Introduction}

In~\cite{zr:29} O.~Zariski computed the fundamental group of the
complement in $\bp^2$ of the discriminant curve of the projection of
a generic cubic surface in $\bp^3$ onto $\bp^2$. Such group turns out to be $\bz/2\bz*\bz/3\bz$
and this is related to the fact that the equation of such a curve is of the type 
$f_2^3+f_3^2=0$, where $f_m$ is a homogeneous polynomial of degree~$m$
in three variables. Even if not all the details are fully justified in
that paper, the result is true and the techniques therein are behind the well-known 
Zariski-van Kampen method~\cite{zr:29,vk:33}.

Following these ideas, M.~Oka~\cite{oka:75} proved that the fundamental group
of the complement of the projective curve with equation $(x^p+y^p)^q+(y^q+z^q)^p=0$
is isomorphic to $\bz/p\bz*\bz/q\bz$, for $p,q$ coprime. In order to obtain this result
the author is forced to perform long computations as a result of using the Zariski-van Kampen method
(see~\cite{nem:87} for another approach given by A.~N{\'e}methi). By standard arguments, the same 
computations used by Oka are true for the complement of $f_p^q+f_q^p=0$, which will be called from 
now on a \emph{$(p,q)$-torus curve}, for generic~$f_p$ and $f_q$: there is an isotopy from any generic 
$(p,q)$-torus curve to Oka's curve. Therefore the fundamental group of the complement of 
$f_p^q+f_q^p=0$ is isomorphic to  $\bz/p\bz*\bz/q\bz$ as well.

In this note, we present a new method to compute fundamental groups of complements of curves
via Nagata transformations and fibrations from ruled rational surfaces.
This way we give an alternative computation for the fundamental group of generic 
$(p,q)$-torus curves. 

The idea is to use the first part to compute the fundamental group in a degenerated case and then
use orbifold fundamental group techniques to recover the generic case. 

\section{Zariski-Van Kampen Method on ruled surfaces}

We will describe this general method in a particular example using the idea of the classical
Zariski-Van Kampen Method on Hirzebruch surfaces.

Let $N:=2 m+1$ be an odd number coprime with $d:=a+b$, $a,b\in\bn$ coprime.
We consider the plane projective curve $\cC:=\cC_{N,a,b}$ defined by the following
equation:
\begin{equation}\label{eq-curva-particular}
F_{N,a,b}(x,y,z):=x^{a N} y^{b N}+(x^N+y^N+x^m y^m z)^d=0.
\end{equation}
This curve has degree~$d N$. Note that $P:=[0:0:1]$ is a singular point
of $\cC$. Let us consider the pencil of lines passing through $P$.
We denote them by $L_t$, $t\in\bp^1\equiv\bc\cup\{\infty\}$, where
$$
L_t:=
\begin{cases}
\{y-t x=0\}&\text{ if }t\in\bc\\
\{x=0\}&\text{ if }t=\infty.
\end{cases}
$$
The lines $L_0$ and $L_\infty$ intersect $\cC$ only at $P$.
Since the multiplicity of $(\cC,P)$ equals $2 m d=(N-1) d$,
a generic line $L_t$ intersect $\cC$ at $N d - (N-1) d=d$ points outside $P$. 
The tangent cone of $(\cC,P)$ consists of the two lines $L_0$ and $L_\infty$. 
In order to study this singularity we perform
a blowing-up at $P$. With suitable charts
the local equation of the strict transform of the branches
tangent to $L_0$ are of the form:
\begin{equation}\label{eq-1st-blow-up}
x^{a N} y^{d}+(x^N y+y+x^m)^d=0,
\end{equation}
which is tangent to $y=0$ (as far as $m>1$), the equation of the exceptional divisor on this chart.
Let $y_1:=y+x^m$. Then~\eqref{eq-1st-blow-up} becomes:
\begin{equation}\label{eq-1st-blow-up-change}
x^{a N} (y_1-x^m)^{d}+(y_1+x^N y_1-x^{N+m})^d=0.
\end{equation}
Looking at its Newton polygon, we deduce that such a singularity
is topologically equivalent to $y_1^d+x^{a N+ m d}=0$.
Since $\gcd(d,a N+m d)=1$, it is irreducible.

Something analogous occurs when considering the branches at the infinitely 
near point associated with the tangent direction~$L_\infty$. 
Blowing-down, we can describe the topological type of the original 
singularity~$(\cC,P)$.

\begin{lema}\label{lema-primera-expl}
The singularity $(\cC,P)$ has two branches. The branch
tangent to $L_0$ is of type $(a N+m d, a N+ (m+1) d)$ and the branch
tangent to $L_\infty$ is of type $(b N+ (m+1) d, b N+m d)$.
\end{lema}

Using Riemann-Hurwitz arguments, any other line through $P$ intersects $\cC$ 
transversally outside $P$ and in particular $P$ is the only singular point of~$\cC$.

We want to compute the fundamental group of $\bp^2\setminus\cC$ using a generalized
Zariski-van Kampen method where $P$ is the projection point. For the classical 
Zariski-van Kampen method the projection point is a generic point outside~$\cC$.
In our case, not only $P\in\cC$ but also the tangent cone of $(\cC,P)$ consists of two lines;
hence, for any choice of line at infinity we should deal with vertical asymptotes
and this is a strong technical problem. In order to deal with these issues, we are going to
perform Nagata's elementary operations to some ruled surfaces. Since it will be more
useful for our purposes, we replace Nagata's elementary operations by a sequence
of blowing-ups and blowing-downs.

Let $\sigma_0:\Sigma_1\to\bp^2$ be the blowing-up of $\bp^2$ at $P$ and denote by 
$E:=\sigma_0^{-1}(P)$ the exceptional divisor. The divisor $E$ is a $(-1)$-curve which is a 
section of the ruling~$\Pi_1:\Sigma_1\to\bp^1$.

\begin{cnv}
For the sake of simplicity, given a blowing-up, the strict transform of a curve will keep the same notation.
\end{cnv}
\begin{figure}[ht]
\centering
\includegraphics[scale=1]{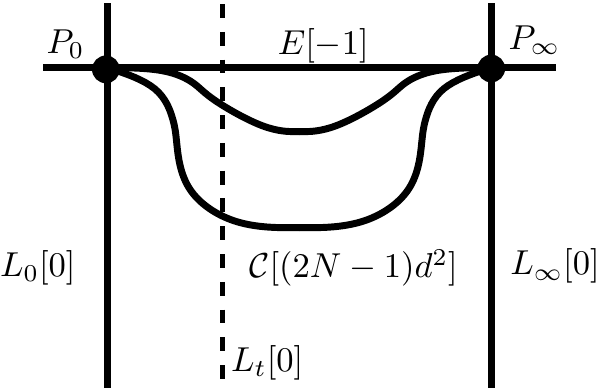}
\caption{Projection after the blowing-up of $P$.}
\label{fig-reglada1}
\end{figure}
One has the following properties:
\begin{enumerate}
\enet{\rm($\Sigma_1$\arabic{enumi})} 
\item $\sigma_0^{-1}(\cC)=\cC\cup E$.
\item $L_0$ and $L_{\infty}$ are the only fibers of $\Pi_1$ which are non transversal to $\cC$ (and to $\cC\cup E$).
\end{enumerate}
Figure~\ref{fig-reglada1} describes the standard ruling $\Pi_1$ of $\Sigma_1$. The number in brackets after a divisor 
represents its self-intersection. This is not yet a good model for any Zariski-van Kampen based method since the 
curve $\cC$ intersects the negative curve $E$. Let $P_t:=L_{t}\cap E$. Following Lemma~\ref{lema-primera-expl} and 
\eqref{eq-1st-blow-up}-\eqref{eq-1st-blow-up-change}, we deduce that $(\cC,P_\infty)$ is a singular point of type 
$(a N+m d, d)$, the curve $E$ is tangent and $(\cC\cdot E)_{P_\infty}=m d$. 

Analogously $(\cC,P_0)$ is a singular point of type $(b N+m d, d)$, the curve $E$ is tangent
and $(\cC\cdot E)_{P_0}=m d$. In order to separate $E$ and $\cC$ we perform $m$ blowing-ups at $P_\infty=P_\infty^0$ 
(and the following $(m-1)$ infinitely near points $P_\infty^j$ at $E$, $j=1,\dots,m-1$) where 
$P_\infty^j:=\cC\cap E_\infty^j$ and $E_\infty^j$ is the exceptional divisor obtained after blowing up $P_\infty^{j-1}$.
Also, note that the multiplicity of $\cC$ at $P_\infty^j$ is $d$.
Note that the point $P_ \infty^m$ is not on $E$ anymore. One repeats this procedure for $P_0=P_0^0$. 
This way a surface~$\widehat{\Sigma}$ is obtained as shown in Figure~\ref{fig-reglada-expl}.
\begin{figure}[ht]
\centering
\includegraphics[scale=1]{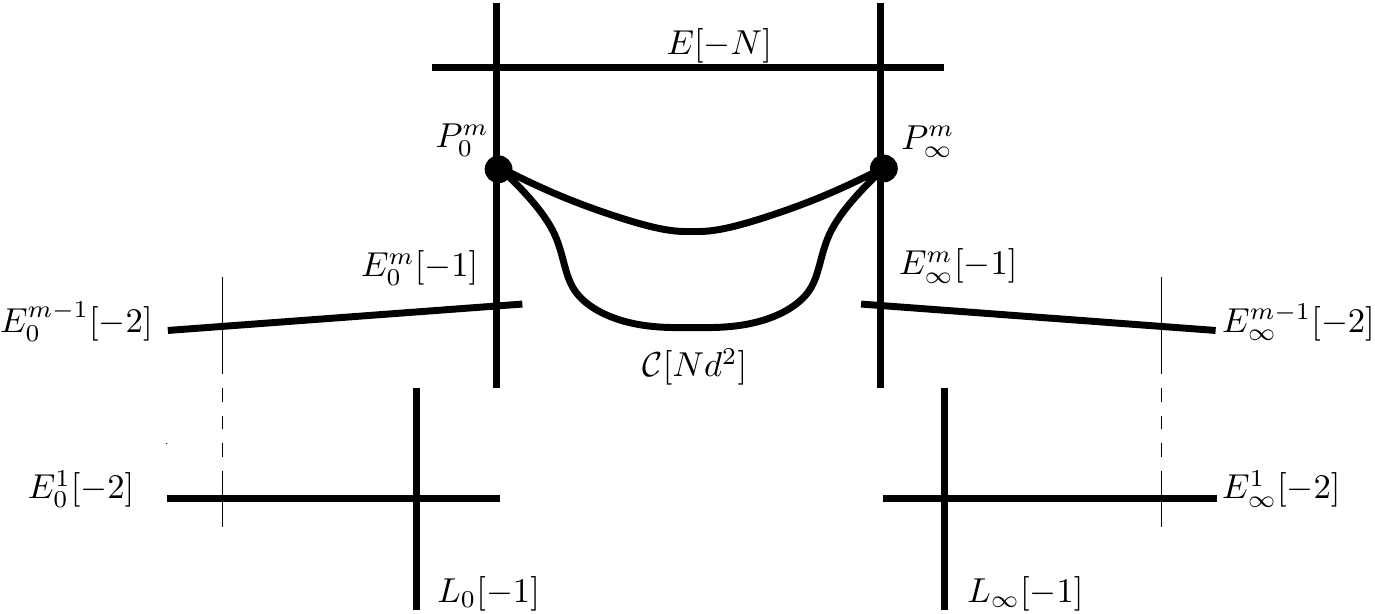}
\caption{The surface $\widehat{\Sigma}$.}
\label{fig-reglada-expl}
\end{figure}
Now, we can blow down the curves $L_\infty,E_\infty^1,\dots,E_\infty^{m-1}$ and $L_0,E_0^1,\dots,E_0^{m-1}$
in order to obtain a ruled surface $\Sigma_N$ where $\cC$ is disjoint to $E$.
The singularity type of $(\cC,P^m_\infty)$ is $(aN,d)$ and it is transversal to $E_\infty^m$. 
Analogously, the singularity type of $(\cC,P^m_0)$ is $(bN,d)$ and it is transversal to $E_0^m$.

\begin{figure}[ht]
\centering
\includegraphics[scale=1]{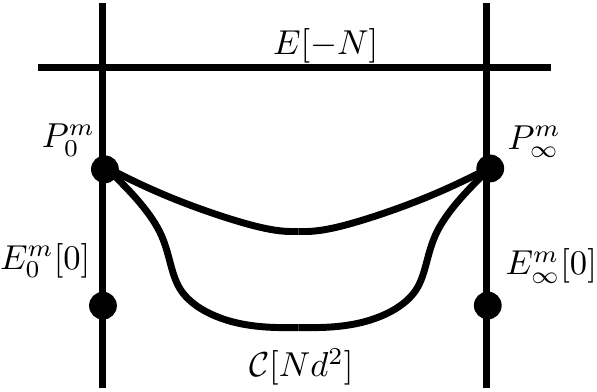}
\caption{The situation on $\Sigma_N$.}
\label{fig-reglada2}
\end{figure}

\begin{prop}
The monodromy action on $\Sigma_N$ is given as follows. Let $L_t$ be a generic fiber.
We can choose meridians $\mu_0$, $\mu_\infty$, and $\mu_t$ on $\bp^1\setminus\{p_0,p_\infty,p_t\}$ 
around the images $p_0,p_\infty,p_t$ of the fibers $E_0^m,E_\infty^m,L_t$ of $\Pi_N:\Sigma_N\to\bp^1$, 
such that $\mu_0\cdot\mu_\infty\cdot \mu_t=1\in\pi_1(\bp^1\setminus\{p_0,p_\infty,p_t\})$. For these 
meridians, the braid monodromy is given by
\begin{equation}
\label{eq-braid-mon}
\mu_0\mapsto \beta_0=(\sigma_{d-1}\cdot\ldots\cdot\sigma_1)^{a N},\quad
\mu_\infty\mapsto \beta_\infty=(\sigma_{d-1}\cdot\ldots\cdot\sigma_1)^{b N},
\end{equation}
where the $\sigma_1,\dots,\sigma_{d-1}$ are half twists generating the braid group $\bb_d$ on $L_t$ based 
on the set $L_t\cap \cC$.
\end{prop}

\begin{proof}
The braid monodromy is defined over $\bp^1\setminus\{p_0,p_\infty,p_t\}$ (which is like a two-punctured $\bc$).
We need to take out a fiber, say $L_t$, in order to have a trivialization that allows one to define the braids.
From the topological type of the singularity at $P_0^m$ we know that the image $\beta_0$ of $\mu_0$
can be chosen to be a conjugate of $(\sigma_{d-1}\cdot\ldots\cdot\sigma_1)^{a N}$,
hence after an appropriate choice of the generators $\sigma_i$, the image $\beta_0$ can be set to be exactly 
$(\sigma_{d-1}\cdot\ldots\cdot\sigma_1)^{a N}$. Once the generators $\sigma_i$ are fixed and using the topological
type of the singularity $P_\infty^m$, the image $\beta_\infty$ of $\mu_\infty$ will be a conjugate of 
$(\sigma_{d-1}\cdot\ldots\cdot\sigma_1)^{b N}$. Since $L_t$ is generic, we are in $\Sigma_N$ and we are avoiding 
the negative section $E$, we have that the image of $\mu_t^{-1}=\mu_\infty\cdot\mu_0$ is $\Delta^{2 N}$, where 
$\Delta^2=(\sigma_{d-1}\cdot\ldots\cdot\sigma_1)^{d}$ is the positive generator of the center of $\bb_d$, 
see~\cite[Lemma~2.1]{khku:04}. Hence, $\beta_\infty=(\sigma_{d-1}\cdot\ldots\cdot\sigma_1)^{b N}$ is as in the statement.
\end{proof}

Note that 
\begin{equation}
\label{eq-p1}
\bp^2\setminus(\cC\cup L_0\cup L_\infty)\cong
\Sigma_1\setminus(\cC\cup E\cup L_0\cup L_\infty)\cong
\Sigma_N\setminus(\cC\cup E\cup E_0^m\cup E_\infty^m).
\end{equation}

\begin{lema}
The meridians of $L_\infty$ and $E_\infty^m$ (resp. $L_0$ and $E_0^m$) are conjugate in $\widehat{\Sigma}$. 
As a consequence,
$$
\pi_1(\bp^2\setminus\cC)\cong\pi_1(\Sigma_N\setminus(\cC\cup E)).
$$
\end{lema}

\begin{proof}
Let us consider the sequence of blowing-ups used to obtain Figure~\ref{fig-reglada-expl} from \ref{fig-reglada2}.
Note that if we blow up a smooth point on a divisor $D$, a meridian of the exceptional component equals
a meridian of $D$. This gives the first statement. The second part is a consequence of~\eqref{eq-p1} and the 
following well-known fact: let $S$ be a quasi-projective surface and let $A$ be an irreducible curve in $S$, then
the map $\pi_1(S\setminus A)\to\pi_1(S)$ is surjective and its kernel is normally generated by a meridian of~$A$. 
\end{proof}

\begin{thm}\label{prop-curva-particular}
$\pi_1(\bp^2\setminus\cC)\cong\bz/d\bz*\bz/N\bz$.
\end{thm}

\begin{proof}
The Zariski-van Kampen method on $\Sigma_N\setminus(\cC\cup E)$ provides a way to obtain a presentation of
$\pi_1(\Sigma_N\setminus(\cC\cup E))$ as follows. Let us define $\beta:=(\sigma_{d-1}\cdot\ldots\cdot\sigma_1)^{N}$.
According to this, $\beta_0=\beta^a$ and $\beta_\infty=\beta^b$ in~\eqref{eq-braid-mon}. Then, 
\begin{equation}
\label{eq-main-thm}
\pi_1(\Sigma_N\setminus(\cC\cup E))=\langle\mu_1,\dots,\mu_d\mid \mu_i=\mu_i^{\beta_0}, \mu_i=\mu_i^{\beta_\infty},
1\leq i\leq d,(\mu_d\cdot\ldots\cdot\mu_1)^N=1\rangle.
\end{equation}
Since $\gcd(a,b)=1$, the first two sets of relations in~\eqref{eq-main-thm} can be replaced by 
$\mu_i=\mu_i^{\beta}$, $1\leq i\leq d$. Since by hypothesis $\gcd(d,N)=1$, these are exactly the relations of the 
fundamental group of torus knot of type $(d,N)$, which is given by $A^d=B^N$ for 
$B=\mu_d\cdot\ldots\cdot\mu_1$. The last relation in~\eqref{eq-main-thm} implies $B^N=1$, and hence the result follows.
\end{proof}

\section{Torus curves}

\begin{dfn}
A \emph{$(p,q)$-torus curve}, $\gcd(p,q)=1$, is a curve $\cC$ admitting an equation 
$f_p^q+f_q^p=0$, where $f_m$ is a homogeneous polynomial of degree~$m$
in three variables. We say that $\cC$ is a \emph{generic torus curve} if
the curves of equation $f_p=0$ and $f_q=0$ intersect transversally at $p q$ distinct points
and they are the only singular points of $\cC$.
\end{dfn}

The following result is straightforward.

\begin{prop}\label{prop-toric}
Let $\cC_0,\cC_1$ be $(p,q)$-torus curves, $\cC_1$ being generic. Then, there is a continuous
path $\gamma:[0,1]\to\bp_{p q}$ ($\bp_d$ is the projective space of all curves of degree~$d$)
such that if we denote $\cC_t:=\gamma(t)$, then $\cC_t$ is a generic $(p,q)$-torus curve,
$\forall t\in(0,1]$.
\end{prop}

We apply the following result which can be found in Dimca's book~\cite{dim:li} or Zariski's paper~\cite{zr:29}.

\begin{prop}\label{prop-camino}
Let $\gamma:[0,1]\to\bp_{d}$ be a continuous map and denote $\cC_t:=\gamma(t)$. Assume that
$\cC_t$ are equisingular for all $t\in(0,1]$. Then
\begin{enumerate}
\enet{\rm(D\arabic{enumi})} 
\item\label{D1} If $\cC_0$ is also equisingular, then $\cC_0$ and $\cC_1$ are isotopic and hence
$\pi_1(\bp^2\setminus\cC_0)\cong\pi_1(\bp^2\setminus\cC_1)$.
\item If $\cC_0$ is reduced, then there is a natural epimorphism
\begin{equation}\label{eq-epi}
 \pi_1(\bp^2\setminus\cC_0)\mathop{\longtwoheadrightarrow}^{\rho(\cC_1,\cC_0)}\pi_1(\bp^2\setminus\cC_1)
\end{equation}
defined as follows
$$
\pi_1(\bp^2\setminus\cC_0)\mathop{\longleftarrow}^{i_*}_{\cong}
\pi_1(\bp^2\setminus R(\cC_0))\mathop{\longtwoheadrightarrow}^{i_*}
\pi_1(\bp^2\setminus\cC_t) \mathop{\longleftrightarrow}^{\eqref{D1}}_{\cong}
\pi_1(\bp^2\setminus\cC_1),
$$
where $R(\cC)$ means a small enough regular neighborhood and the maps denoted by $i$ are inclusions.
\end{enumerate}
\end{prop}

As a direct consequence of Theorem~\ref{prop-curva-particular} and Propositions~\ref{prop-toric}
and~\ref{prop-camino} the following holds.

\begin{cor}
There is a natural epimorphism $\bz/p\bz*\bz/q\bz\twoheadrightarrow\pi_1(\bp^2\setminus\cC)$
for any generic $(p,q)$-torus curve~$\cC$.
\end{cor}

The final step involves orbifold maps. Let us recall the definitions.

\begin{dfn}
An \emph{orbifold} $X_\varphi$ is a quasi-projective Riemann surface~$X$
with a function $\varphi:X\to\bn$ taking value~$1$ outside a finite
number of points.
\end{dfn}

\begin{dfn}\label{dfn-group-orb}
For an orbifold $X_\varphi$, let $p_1,\dots,p_n$ be the points such that
$m_j:=\varphi(p_j)>1$. Then, the \emph{orbifold fundamental group} of $X_\varphi$ is
$$
\pi_1^{\text{\rm orb}}(X_\varphi):=\pi_1(X\setminus\{p_1,\dots,p_n\})/\langle\mu_j^{m_j}=1\rangle,
$$
where $\mu_j$ is a meridian of $p_j$. We denote $X_\varphi$ by $X_{m_1,\dots,m_n}$.
\end{dfn}

\begin{dfn}
Let $X_\varphi$ be an orbifold and $Y$ a smooth algebraic variety.
A dominant algebraic morphism $\rho:Y\to X$ defines an \emph{orbifold morphism} $Y\to X_\varphi$
if for all $p\in X$, the divisor $\rho^*(p)$ is a $\varphi(p)$-multiple.
\end{dfn}

\begin{prop}[\cite{cko,AC-prep}]\label{prop-orb}
Let $\rho:Y\to X$ define an orbifold morphism $Y\to X_\varphi$. Then $\rho$ induces a morphism
$\rho_*:\pi_1(Y)\to\pi_1^{\text{\rm orb}}(X_\varphi)$. Moreover, if the generic fiber is connected, then
$\rho_*$ is surjective.
\end{prop}

\begin{prop}\label{prop-orb-pq}
Let $\cC$ be a  $(p,q)$-torus curve~$\cC$. Then there exists 
a natural epimorphism $\rho_\cC:\pi_1(\bp^2\setminus\cC)\to\bz/p\bz*\bz/q\bz$.
Moreover, if $\cC_0,\cC_1$ are $(p,q)$-torus curves, $\cC_1$ generic, then
the map $\rho(\cC_1,\cC_0)$ in \eqref{eq-epi} satisfies
\begin{equation}
\begin{tikzpicture}[description/.style={fill=white,inner sep=2pt},baseline=(current bounding box.center)]
\matrix (m) [matrix of math nodes, row sep=3em,
column sep=2.5em, text height=1.5ex, text depth=0.25ex]
{ \pi_1(\bp^2\setminus\cC_0)& & \pi_1(\bp^2\setminus\cC_1) \\
& \bz/p\bz*\bz/q\bz & \\ };
\path[->>,>=angle 90](m-1-1) edge node[auto] {$  \rho(\cC_1,\cC_0) $} (m-1-3);
\path[->>,>=angle 90](m-1-1)  edge node[auto,swap] {$  \rho_{\cC_0} $}  (m-2-2);
\path[->>,>=angle 90](m-1-3)  edge node[auto] {$  \rho_{\cC_1} $}  (m-2-2);
\end{tikzpicture}
\end{equation}
\end{prop}

\begin{proof}
For a curve $\cC$ the map $\rho_\cC$ comes from the orbifold map 
$\bp^2\setminus\cC\to \bp^1_{p,q}\setminus\{[1:-1]\}$ given
by $[x:y:z]\mapsto[f_p^q:f_q^p]$. The genericity guarantees that the
generic fiber is irreducible. The last statement comes from the fact
that the orbifold map can be put in a family.
\end{proof}

\begin{obs}
There is a slight ambiguity in Proposition~\ref{prop-orb-pq}, since a torus
curve may admit several decompositions. The last statement is true 
if the deformation respects the decompositions.
\end{obs}

Let $\cC_0$ be a curve with equation \eqref{eq-curva-particular}, for $N=p$, $a+b=d=q$.
Then, by Proposition~\ref{prop-curva-particular}, the map $\varphi:=\rho_{\cC_1}\circ\rho(\cC_1,\cC_0)$
defines an epimorphism 
\begin{equation}\label{eq-epi1}
\bz/p\bz*\bz/q\bz\twoheadrightarrow\bz/p\bz*\bz/q\bz. 
\end{equation}

\begin{thm}
The fundamental group of any generic $(p,q)$-curve is isomorphic to the free product  $\bz/p\bz*\bz/q\bz$.
\end{thm}

\begin{proof}
It is enough to prove that \eqref{eq-epi1} is an isomorphism which is a straightforward result from group theory.
\end{proof}

\providecommand{\bysame}{\leavevmode\hbox to3em{\hrulefill}\thinspace}
\providecommand{\MR}{\relax\ifhmode\unskip\space\fi MR }
\providecommand{\MRhref}[2]{%
  \href{http://www.ams.org/mathscinet-getitem?mr=#1}{#2}
}
\providecommand{\href}[2]{#2}

\end{document}